\documentclass[12pt]{amsart}
\usepackage{amssymb}
\usepackage[margin=1in]{geometry}

\newtheorem{theorem}{Theorem}[section]
\newtheorem{corollary}[theorem]{Corollary}
\newtheorem{prop}[theorem]{Proposition}
\newtheorem{lemma}[theorem]{Lemma}
\newtheorem{proposition}[theorem]{Proposition}

\theoremstyle{definition}

\newtheorem{example}[theorem]{Example}

\newtheorem{remark}[theorem]{Remark}

\newcommand{\fn}{\!:}
\newcommand{\C}{\mathbb C}
\newcommand{\R}{{\mathbb R}}

\newcommand{\lla}{\left\langle}
\newcommand{\rra}{\right\rangle}
\newcommand{\mc}{\mathcal}

\newcommand{\tn}{\textnormal}

\newcommand{\Z}{\mathbb Z}

\newcommand{\F}{\mathbb F}

\begin{document}

\title[On Exotic Group C*-algebras]
{On Exotic Group C*-algebras} 

\date{\today}

\author[Zhong-Jin Ruan and Matthew Wiersma]{Zhong-Jin Ruan$^1$ and Matthew Wiersma$^2$}

\address{Zhong-Jin Ruan: Department of Mathematics, University  of Illinois at Urbana-Champaign, Urbana, IL 61801, USA}
\email{z-ruan@illinois.edu}

\address{Matthew Wiersma: Department of Pure Mathematics, University  of Waterloo, Waterloo, ON canada N2L 3G1}
\email{mwiersma@uwaterloo.ca}

\keywords{Amenability, Factorization Property,  and Local Properties of Groups and Group C*-algebras}
\thanks{2010 \it{Mathematics Subject Classification:}
\rm{Primary 22D25; Secondary 46L06, 43A35}}
\thanks{$^1$ The first author was  partially supported by the Simons Foundation.}
\thanks{$^2$ The second author was partially supported by an NSERC Postgraduate Scholarship.}


\begin{abstract}
Let $\Gamma$ be a discrete group. A C*-algebra $A$ is an exotic C*-algebra (associated to $\Gamma$) if there exist proper surjective C*-quotients $C^*(\Gamma)\to A\to C^*_r(\Gamma)$ which compose to the canonical quotient $C^*(\Gamma)\to C^*_r(\Gamma)$. In this paper, we show that a large class of exotic C*-algebras have poor local properties. More precisely, we demonstrate the failure of local reflexitity, exactness, and local lifting property. Additionally, $A$ does not admit an amenable trace and, hence, is not quasidiagonal and does not have the WEP when $A$ is from the class of exotic C*-algebras defined by Brown and Guentner (see \cite{bg}). In order to achieve the main results of this paper, we prove a  result which implies the factorization property for the  
class of discrete groups which are algebraic subgroups of locally compact amenable groups.
  
\end{abstract}
\maketitle

\section{Introduction} \label{section:intro}


There are two natural group C*-algebras associated to each discrete group $\Gamma$: the full group C*-algebra $C^*(\Gamma)$ and the reduced group C*-algebra $C^*_r(\Gamma)$.
The reduced group C*-algebra is always a canonical quotient of the full group C*-algebra and, further, the group $\Gamma$ is amenable if and only if this quotient is injective.
In particular, this implies the natural quotient map
$ C^*(\Gamma) \to C^*_r(\Gamma)$ is proper (i.e., non-injective) when $\Gamma$ is non-amenable.
In this case,   we are interested in understanding
intermediate C*-algebras $A$ which sit properly between $ C^*(\Gamma) $ and $ C^*_r(\Gamma) $ in the sense that there exist proper surjective C*-quotients 
\[
C^*(\Gamma) \to A \to C^*_r(\Gamma)
\]
whose composition is the canonical quotient from $C^*(\Gamma)$ onto $C^*_r(\Gamma)$. We call such a C*-algebra $A$ an {\em exotic C*-algebra}
(associated with $\Gamma$) and note that $A$ is a group C*-algebra in the sense that there exists a unitary representation $\pi: \Gamma \to B(H)$ such that 
$ A = C^*_\pi(\Gamma) = \mbox{span}\{\pi(s) : s\in \Gamma\}^{-\|\cdot\|}\subseteq B(H)$.

Recently,  exotic C*-algebras  
have been studied by  Brown-Guentner (see \cite{bg})  and Wiersma (see \cite{w1}) 
for discrete groups and by Kyed-So\l{}tan (see \cite{ks}) and Brannan-Ruan (see \cite{br})
for discrete quantum groups.
In \cite{bg}, Brown and Guentner introduce an interesting class of 
C*-algebras associated to algebraic ideals $D$ of $\ell^\infty(\Gamma)$. These are called {\it $D$-C*-algebras} and are denoted by $C^*_D(\Gamma)$.
Since  for any $1 \le p \leq  \infty$, $\ell^p(\Gamma)$ is an ideal in $\ell^\infty(\Gamma)$,  we can consider $\ell^p$-C*-algebras
$C^*_{\ell^p}(\Gamma)$. As the C*-algebra $C^*_{\ell^\infty}(\Gamma)$ is always the full group C*-algebra $C^*(\Gamma)$ and $C^*_{\ell^p}(\Gamma)$ is always the reduced group C*-algebra $C^*_r(\Gamma)$ for $1\leq p\leq 2$ (see \cite{bg}),
we are mainly interested in the case of $2 < p < \infty$.

It is shown by Brown and Guentner (see \cite{bg}) and Okayasu (see \cite{o})  that 
for the non-commutative free group ${\mathbb F}_d$ of $d$-generators, 
$C^*_{\ell^p}({\mathbb F}_d)$ are exotic and distinct
for all $2 < p < \infty$.
It follows that if $\Gamma$ is a discrete group containing a copy of $\F_2$ (e.g., $SL_n(\Z)$ for $n\geq 2$),  then  $C^*_{\ell^p}(\Gamma)$ are also exotic and distinct for each $2 < p < \infty$ (see \cite{w1}).
The results in this paper apply to many interesting examples of $D$-C*-algebras including $C^*_{\ell^p}(\Gamma)$ when $\Gamma$ contains a noncommutative free group and $2<p<\infty$.

The main results of this paper apply to a much broader class of exotic C*-algebras than simply those of the form $C^*_{D}(\Gamma)$. Bekka demonstrates in \cite{b} that there are many residually finite discrete groups $\Gamma$ whose full group C*-algebra $C^*(\Gamma)$ is not residually finite dimensional. Examples include $SL_n(\Z)$ for $n\geq 3$ and $Sp_n(\Z)$ for $n\geq 2$. In this case, the group C*-algebra $C^*_{\mc F}(\Gamma)$ arising from the finite dimensional representations of $\Gamma$ is exotic.
These C*-algebras are also included in the class of exotic group C*-algebras considered in this paper.
Besides these cases, it also includes the construction of Bekka-Kaniuth-Lau-Schlichting (see \cite {bkls}).

Recall that the reduced group C*-algebra $C^*_r(\Gamma)$ often has some kind of approximation property.
For instance, a discrete group $\Gamma$ is amenable 
if and only if its  reduced  group C*-algebra $C^*_r(\Gamma)$ is nuclear.
For non-amenable groups like ${\mathbb F}_d$,  hyperbolic groups,  ${\mathbb Z}^2 \rtimes SL_2({\mathbb Z})$, and $SL_3({\mathbb Z})$, their reduced group C*-algebras have the completely contractive approximation property (CCAP), completely bounded approximaton property (CBAP), operator approximation property (OAP), and exactness, respectively.
It is natural to ask whether exotic group C*-algebras, such as
$C^*_{\ell^p}({\mathbb F}_d)$ for $2 < p < \infty$, have any such approximation properties.
Our results strongly imply this is not the case.
Even though the main results in this paper are dealing with group C*-algebras, 
the major techniques we apply in the proofs are largely motivated from abstract harmonic analysis.

This paper is organized as follows. The following section recalls some preliminary notation and results. In section 3, we prove a result which implies the factorization property for discrete groups which embed algebraically into amenable locally compact groups (Theorem \ref{factorization property}). This result proves invaluable to us throughout the rest of the paper. The main results of this paper are contained in section 4. Here the class of exotic C*-algebras we consider is defined and is shown to never be locally reflexive (Theorem \ref{nonlocal}), thus, not exact (Theorem \ref{nonexact}), and never to have the local lifting property (Theorem \ref{llp}). In the section 5, we conclude the paper by remarking that exotic $D$-C*-algebras fail to admit additional desirable properties for general discrete groups. More specifically, they are shown to fail to admit an amenable trace (Proposition \ref{trace}) and, so, are not quasidiagonal and do not have the weak expectation property (Corollary \ref{quasi and wep}). Since amenable groups admit no exotic group C*-algebras, we will primarily be studying C*-algebras associated with non-amenable groups throughout this paper.


\section{Preliminaries}

In this section we recall $D$-representations and their associated C*-algebras as defined by Brown and Guentner (see \cite{bg}). These C*-algebras are an important source of examples throughout this paper. 
At the end of this section we remark on a way of viewing tensor products of group $C^*$-algebras which proves useful to us in this paper.

\subsection*{$D$-representations and C*-algebras}

Let $\Gamma$ be a discrete group and $D$ an algebraic ideal of $\ell^\infty(\Gamma)$. A (unitary) representation $\pi\fn \Gamma\to \mc B(H)$ is a {\it $D$-representation} if there exists a dense subspace $H_0$ of $H$ such that the coefficient function $s\mapsto \lla \pi(s)\xi,\xi\rra$ is an element of $D$ for every $\xi$ in the dense subspace $H_0$. Note that if $H_0$ is such a subspace for a $D$-representation $\pi\fn \Gamma\to B(H)$ and $\sigma\fn \Gamma\to B(K)$ is any other representation of $\Gamma$, then $s\mapsto \lla (\pi\otimes\sigma)(s)\xi\otimes\eta,\xi\otimes\eta\rra\in D$ for every $\xi\in H_0$ and $\eta\in K$. Hence, it follows from the polarization identity that $\pi\otimes\sigma$ is also a $D$-representation (see \cite[Remark 2.4]{bg}). We also note that the direct sum of an arbitrary number of $D$-representations is again a $D$-representation (see \cite[Remark 2.5]{bg}) and, consequently, we can build a $D$-representation $\pi_D$ such that $\pi_D$ extends to a faithful $*$-representation of $C^*_D(\Gamma)$ by taking a large enough direct sum of $D$-representations.

We can define a C*-seminorm $\|\cdot\|_{D}$ on the group ring $\C[\Gamma]$ by
$$\|x\|_{D}=\sup\{\|\pi(x)\|: \pi\tn{ is a $D$-representation of }\Gamma\}.$$
Observe that if $\Gamma$ admits a $D$-representation $\pi$, then $\|\cdot\|_D$ is faithful on $\C[\Gamma]$ and it dominates the reduced C*-norm since $\pi\otimes \lambda$ is a $D$-representation which is unitarily equivalent to an amplification of $\lambda$ by Fell's absorption principle.
The {\it $D$-C*-algebra} $C^*_{D}(\Gamma)$ is the completion of the group ring $\C[\Gamma]$ with respect to $\|\cdot\|_{D}$.

The two classes of ideals $D$ which are most predominately studied in the literature are when $D=c_0$ ($=c_0(\Gamma)$) or $D=\ell^p$ ($=\ell^p(\Gamma)$) for $1\leq p<\infty$. Brown and Guentner observed that a discrete group $\Gamma$ has the Haagerup property if and only if $C^*_{c_0}(\Gamma)=C^*(\Gamma)$ canonically (see \cite[Corollary 3.4]{bg}) and that $\Gamma$ is amenable if and only if the trivial representation $1_\Gamma$ extends to a $*$-representation of $C^*_{\ell^p}(\Gamma)$ for some $1\leq p<\infty$ (see \cite[Remark 2.13]{bg}). Brown and Guentner also observed $\|\cdot\|_{\ell^p}=\|\cdot\|_r$ for each $1\leq p\leq 2$ (see \cite[Proposition 2.11]{bg}), but showed that for a noncommutative free group $\F_d$ on finitely many generators, there exists $2<p<\infty$ such that $\|\cdot\|_{\ell^p}$ is distinct from both the full and reduced group C*-norms (see \cite[Proposition 4.4]{bg}). Okayasu substantially improved upon this result by showing that these C*-norms $\|\cdot\|_{\ell^p}$ are distinct for every $2\leq p\leq\infty$ for $\F_d$ (see \cite[Corollary 3.2]{o}). It was later observed by Wiersma (see \cite[Theorem 2.4]{w1}) that if $H$ is a subgroup of a discrete group $\Gamma$, then $C^*_{\ell^p}(H)$ canonically embeds into $C^*_{\ell^p}(\Gamma)$. Consequently, if $\Gamma$ contains a copy of $\F_2$, then the C*-norms $\|\cdot\|_{\ell^p}$ on $\C[\Gamma]$ are distinct for every $2\leq p\leq\infty$. We also note that $\ell^p$-C*-algebras can also be exotic group C*-algebras for non-amenable groups $\Gamma$ not containing a copy of $\F_2$. For example, if $\Gamma$ is a non-amenable Coxeter group, then there exists $2<p<\infty$ such that $C^*_{\ell^p}(\Gamma)$ is an exotic C*-algebra (see \cite[Remark 4.5]{bg}).

\subsection*{Additional Remark}

We finish the preliminaries by making a brief remark about tensor products of group C*-algebras. Let $\Gamma$ be a discrete group. Then the algebraic tensor product $\C[\Gamma]\odot\C[\Gamma]$ is equal to $\C[\Gamma\times\Gamma]$ via the identification of $s\otimes t$ with $(s,t)$ for $s,t\in\Gamma$. For representations $\pi$ and $\sigma$ weakly containing $\lambda$ this, in particular, allows us to view C*-completions of $C^*_{\pi}(\Gamma)\odot C^*_{\sigma}(\Gamma)$ as being C*-completions of $\C[\Gamma\times\Gamma]$ -- an identification which will extremely valuable to us.

\section{The Factorization Property}

A locally compact group $G$ is said to have the {\it factorization property} in the sense of Kirchberg (see \cite{k}) if the representation $\sigma_{\lambda_G,\rho_G}\fn G\times G\to B(L^2(G))$ defined by 
\[
\sigma_{\lambda_G,\rho_G}(s_1,s_2)\xi(t)=\lambda_G(s_1)\rho_G(s_2)\xi(t)=\xi(s_1^{-1}ts_2)\Delta(s_2)^{1/2}
\]
extends to a $*$-representation of $C^*(G)\otimes_{\min}C^*(G)$. Wassermann showed that if a countable discrete group $\Gamma$ is residually finite, then $\Gamma$ has the factorization property (see \cite{wa}). The following theorem implies the factorization property for a much larger class of discrete groups.

\begin{lemma}\label{bai}
Let $G$ be a locally compact group and $\{f_i\}$ an approximate identity for $L^1(G)$ with $f_i\geq 0$. If $U$ is a neighbourhood of the identity in $G$, then $\int_{G\backslash U} f_i(s)\,ds\to 0$.
\end{lemma}

\begin{proof}
Suppose towards a contradiction that $\int_{G\backslash U} f_i(s)\,ds\not\to 0$ and let $V$ be a compact symmetric (i.e., inverse closed) neighbourhood of the identity such that $V^2\subset U$. Then
\begin{eqnarray*}
\left\|f_i*1_V-1_V\right\|_1 &\geq& \int_{G\backslash V} \left|\int_G f_i(s)1_V(s^{-1}t)\,ds-1_V(t)\right|dt \\
&= & \int_{G\backslash V} \int_V f_i(ts)\,ds\,dt \\
&\geq & m(V) \inf_{s\in V} \int_{G\backslash V} f_i(ts)\,dt \\
&\geq & m(V) \left(\min_{s\in V} \Delta(s)^{-1}\right)\int_{G\backslash U}f_i(t)\,dt\,\not\to\, 0.
\end{eqnarray*}
This contradicts the assumption that $\{f_i\}$ is an approximate identity for $L^1(G)$ and, hence, we conclude that $\int_{G\backslash U} f_i(s)\,ds \to 0$.
\end{proof}

Let $G$ be a  locally compact  group and $\Gamma$ an algebraic subgroup of $G$ endowed with the discrete topology. By algebraic subgroup, we mean a subgroup of $G$ which is not necessarily closed in the topology of $G$.
Then $\lambda_G$ restricted to $\Gamma$ defines a representation
\[
\lambda_G|_\Gamma : \Gamma \to B(L^2(G)).
\]
We let $C^*_{\lambda_G|_{\Gamma}}(\Gamma) = {\rm span}\{\lambda_G|_{\Gamma}(s): s\in \Gamma\}^{- \|\cdot\|}\subseteq B(L^2(G))$
denote the corresponding group C*-algebra.
It is always the case that $C^*_{r}(\Gamma)$ is a canonical quotient of $C^*_{\lambda_G|_\Gamma}$ (see \cite{bkls}) or, equivalently, that $\lambda_\Gamma$ is weakly contained in $\lambda_G|_\Gamma$.

\begin{theorem}\label{factorization property}
Let $G$ be an amenable locally compact  group and $\Gamma$ an algebraic subgroup of $G$ endowed with the discrete topology. Then the representation $\sigma_{\lambda_\Gamma,\rho_\Gamma}\fn \Gamma\times\Gamma\to B(\ell^2(\Gamma))$ defined by $\sigma_{\lambda_\Gamma,\rho_\Gamma}(s,t)=\lambda(s)\rho(t)$ extends to a $*$-representation of $C^*_{\lambda_G|_\Gamma}(\Gamma)\otimes_{\min} C^*_{\lambda_G|_{\Gamma}}(\Gamma)$.
\end{theorem}

\begin{proof}

Since $G$ is amenable, the group von Neumann algebra $L(G)$ is injective. Therefore the multiplication map from $L(G)\odot R(G)$ to $B(L^2(G))$ defined by $a\otimes b\mapsto ab$ is continuous with respect to the minimal C*-norm (see \cite[Lemma 2.1]{el}). Since $\rho_G$ is unitarily equivalent to $\lambda_G$ and $C^*_{\lambda_G|_\Gamma}(\Gamma)$ is contained in $L(G)$, it follows that $\sigma_{\lambda_G|_\Gamma,\rho_G|_\Gamma}$ extends to a $*$-representation of $C^*_{\lambda_G|_\Gamma}(\Gamma)\otimes_{\min}C^*_{\lambda_G|_\Gamma}(\Gamma)$ into $B(L^2(G))$. 


To obtain the result, we need to show that the representation $\sigma_{\lambda_\Gamma,\rho_\Gamma}\fn \Gamma\times\Gamma\to B(\ell^2(\Gamma))$ extends to a $*$-representation of $C^*_{\lambda_G|_\Gamma}(\Gamma)\otimes_{\min}C^*_{\lambda_G|_\Gamma}(\Gamma)$. Towards this goal, define $\tau\fn G\to B(L^1(G))$ by $\tau(s)f(t):=f(s^{-1}ts)\Delta(s)$. In the proof of \cite[Theorem 3]{lr}, Losert and Rendler show that for amenable locally compact groups $G'$, there exists a bounded approximate identity $\{f_i\}$ for $L^1(G')$ with $\|f_i\|_1=1$ and $f_i\geq 0$ such that $\|\tau(s)f_i-f_i\|_1\to 0$ uniformly on compact subsets of $G'$. Let $\{f_i\}\subset L^1(G)$ be such a net for $G$ and, for each index $i$, define $\xi_i=f_i^{1/2}\in L^2(G)$.
Then, for fixed $s\in G$, we observe that
\begin{eqnarray*}
\left\|\sigma_{\lambda_G,\rho_G}(s,s)\xi_i-\xi_i\right\|_2^2 &=& \int_G \left|\xi_i(s^{-1}ts)\Delta(s)^{1/2}-\xi_i(t)\right|^2 dt \\
&\leq & \int_G \left|(\xi_i(s^{-1}ts)\Delta(s)^{1/2})^{2}-(\xi_i(t))^2\right|dt \\
&=& \int_G \left| \tau(s)f_i(t)-f_i(t)\right|\,dt\,\to\, 0.
\end{eqnarray*}
So
\begin{eqnarray*}
\left|\left(\sigma_{\lambda_G,\rho_G}\right)_{\xi_i,\xi_i}(s,s)-1\right| &=& \left|\int_G\xi_i(s^{-1}ts)\Delta(s)^{1/2}\xi_i(t)\,dt-\int_G \xi_i(t)^2\,dt\right| \\
&\leq & \|\sigma_{\lambda_G,\rho_G}(s,s)\xi_i-\xi_i\|_2\|\xi_i\|_2 \,\to\, 0.
\end{eqnarray*}
We will next observe that if $s_1$ and $s_2$ are distinct elements of $G$, then $(\sigma_{\lambda_G,\rho_G})_{\xi_i,\xi_i}(s_1,s_2)\to 0$.

Note that if $s\in G$, then $\{\tau(s)f_i\}$ is an approximate identity for $L^1(G)$ since if $g\in L^1(G)$, then
$$ \lim_i\tau(s)f_i * g= \lim_i\,(\tau(s)f_i-f_i)*g+f_i*g = 0*g+g=g.$$
Let $s_1$ and $s_2$ be distinct elements of $G$, choose an closed neighbourhood $U$ of the identity such that $s_1s_2^{-1}\not \in U$ and let $V$ denote $G\backslash (s_2s_1^{-1}U)$. Then
\begin{eqnarray*}
\int_U\left|\sigma_{\lambda_G,\rho_G}(s_1,s_2)\xi_i(t)\right|^2\,dt &= & \int_U f_i(s_2^{-1}(s_2s_1^{-1}t)s_2)\Delta(s_2)\,dt \\
&=& \int_{s_2s_1^{-1}U}\tau(s_2)f_i(t)\,dt \\
&=& \int_{G\backslash V}\tau(s_2)f_i(t)\,dt\,\to\,0
\end{eqnarray*}
by Lemma \ref{bai} as $V$ is an open neighbourhood of the identity.
So, taking $U$ as above, we observe that
\begin{eqnarray*}
&&\left|\left(\sigma_{\lambda_G,\rho_G}\right)_{\xi_i,\xi_i}(s_1,s_2)\right|\\
&\leq &  \left|\int_{U} \left(\sigma_{\lambda_G,\rho_G}(s_1,s_2)\xi_i(t)\right)\xi_i(t)\,dt\right|+\left|\int_{G\backslash U} \left(\sigma_{\lambda_G,\rho_G}(s_1,s_2)\xi_i(t)\right)\xi_i(t)\,dt\right| \\
&\leq & \left(\int_U \left|\sigma_{\lambda_G,\rho_G}(s_1,s_2)\xi_i(t)\right|^2 dt\right)^{1/2}\|\xi_i\|_2+\|\xi_i\|_2\left(\int_{G\backslash U}f_i(t)\,dt\right)^{1/2}\to  0.
\end{eqnarray*}
Therefore, we have shown that this net of positive definite functions $(\sigma_{\lambda_G |_\Gamma,\rho_G |_\Gamma})_{\xi_i,\xi_i}$ on $\Gamma\times \Gamma$ converges pointwise to the characteristic function $1_{\Delta}$, where $\Delta$ denotes the diagonal of $\Gamma\times \Gamma$. 
This shows that $1_\Delta$ is a positive definite element in  $(C^*_{\lambda_G|_{\Gamma}}(\Gamma)\otimes_{\min}C^*_{\lambda_G|_{\Gamma}}(\Gamma))^*$.
As it is easily seen that $\sigma_{\lambda_\Gamma,\rho_\Gamma}$ is the GNS representation of $1_{\Delta}$, it follows that $\sigma_{\lambda_\Gamma,\rho_\Gamma}$ extends to a $*$-homomorphism of $C^*_{\lambda_G|_{\Gamma}}(\Gamma)\otimes_{\min}C^*_{\lambda_G|_{\Gamma}}(\Gamma)$.



\end{proof}

Note that since $C^*_{\lambda_G|_\Gamma}(\Gamma)\otimes_{\min}C^*_{\lambda_G|_\Gamma}(\Gamma)$ is a canonical quotient of $C^*(\Gamma)\otimes_{\min} C^*(\Gamma)$, this theorem immediately implies that such groups $\Gamma$ have the factorization property. This is a weaker statement than that of Kirchberg in \cite[Lemma 7.3(iii)]{k} which states the same is true when $\Gamma$ embeds into a locally compact group $G$ such that $C^*(G)$ is nuclear (e.g., $GL_n(\R)$ for $n\geq 2$, see \cite[Lemma 7.3(i)]{k}).
 Thom demonstrated a counterexample to this statement of Kirchberg in \cite[\S 3]{thom}. As pointed out by Thom, the statement's proof falsely claims that if we define $f_K=1_K/m(K)^{1/2}\in L^2(G)$ for each compact neighbourhood $K$ of the identity in $G$, then the functions $(\sigma_{\lambda_G,\rho_G})_{f_K,f_K}$ converge pointwise to $1_{\Delta_\Gamma}$. As a concrete counterexample and as a matter of interest in relation to the proof of Theorem \ref{factorization property}, we explicitly show that this condition fails for the group $GL_2(\R)$.

\begin{remark}
Consider the compact neighbourhoods
$$ K_n:=\left\{\begin{bmatrix}
a & b \\
c & d
\end{bmatrix} : a,d\in [1-\tfrac{1}{2n},1+\tfrac{1}{2n}],\, b,c\in [-\tfrac{1}{2n},\tfrac{1}{2n}]\right\}\subset GL_2(\R)$$
of the identity for $n\geq 3$. Let $s=\begin{bmatrix}
2 & 0\\
0 & 1
\end{bmatrix}$ and note that
$$s^{-1}\begin{bmatrix}
a & b \\
c & d
\end{bmatrix}s
=
\begin{bmatrix}
a & b/2 \\
2c & d
\end{bmatrix}.$$
Therefore,
$$ K_n\cap s^{-1}K_ns=\left\{\begin{bmatrix}
a & b \\
c & d
\end{bmatrix} : a,d\in [1-\tfrac{1}{2n},1+\tfrac{1}{2n}],\, b\in [-\tfrac{1}{4n},\tfrac{1}{4n}],\,c\in [-\tfrac{1}{2n},\tfrac{1}{2n}] \right\}.$$
Recall that the Haar measure of $GL_2(\R)$ is given by $dt=\frac{1}{(\det t)^2}dm(t)$ where $m$ denotes Lebesgue measure on $\R^4$. It follows that $m(K_n)$ is asymptotic to $\frac{1}{n^4}$ and $m(K_n\cap (s^{-1}K_n s))$ to $\frac{1}{2n^4}$ as $n\to \infty$. So
$$(\sigma_{\lambda_G,\rho_G})_{f_{K_n},f_{K_n}}(s,s)=\frac{1}{m(K_n)}\int 1_{s^{-1}K_ns}(t)1_{K_n}(t)\,dt\to \frac{1}{2}\neq 1.$$
In particular $(\sigma_{\lambda_G,\rho_G})_{f_{K_n},f_{K_n}}$ does not converge pointwise to $1_\Delta$.
\end{remark}

It may be asked whether, perhaps with more careful choices of $K$ for $GL_2(\R)$, one could find a subnet of $(\sigma_{\lambda_G,\rho_G})_{f_K,f_K}$ converging pointwise to $1_{\Delta}$. The following proposition, giving the converse of Theorem \ref{factorization property}, shows this cannot be the case.

\begin{prop}\label{factorization prop}
Let $G$ be a locally compact group and $G_d$ be the group $G$ endowed with the discrete topology. If $\sigma_{\lambda_{G_d},\rho_{G_d}}$ extends to a $*$-representation of $C^*_{\lambda_G}(G_d)\otimes_{\min}C^*_{\lambda_G}(G_d)$, then $G$ is amenable.
\end{prop}

\begin{proof}
Since the representation $\sigma_{\lambda_{G_d},\rho_{G_d}}: G_d \times G_d \to B(\ell^2(G_d))$ extends to a $*$-representation of $C^*_{\lambda_G}(G_d)\otimes_{\min}C^*_{\lambda_G}(G_d)$, it is weakly contained in   $\lambda_G \times \lambda_G : G_d \times G_d \to B(L^2(G \times G))$ and thus 
 there exists a net of positive definite functions $u_i$ of the form $\sum_{k=1}^n (\lambda_G\times\lambda_G)_{f_k,f_k}$ for $f_k\in L^2(G\times G)$ such that $u_i$  converges pointwise to $1_{\Delta} = (\sigma_{\lambda_{G_d},\rho_{G_d}})_{\delta_e, \delta_e}$. Since the restriction $\lambda_G\times\lambda_G|_\Delta$ is unitarily equivalent to an amplification of $\lambda_\Delta$ by Fell's absorption principle, there exists positive definite functions $v_i$ on $G_d$ of the form $\sum_{k=1}^n (\lambda_G)_{g_k,g_k}$ where $g_k\in L^2(G)$ converging pointwise to $1_{G}$.
So $G$ is amenable by Reiter's weak property $(P_2^*)$ (see \cite[Theorem 1]{be} for a nice proof).
\end{proof}

Finally, we observe that the techniques used in this section give rise to a simple proof that a SIN group $G$ is amenable if and only if $C^*_r(G)$ is nuclear. Recall that a SIN group is a locally compact group $G$ which has a neighbourhood base of the identity consisting of compact neighbourhoods $K$ such that $s^{-1}Ks=K$ for every $s\in G$. 
We note that SIN  groups are  all unimodular, i.e.  $\Delta(s) = 1$ for all $s \in G$.
It is clear that  discrete groups are  all SIN.
This equivalence between nuclearity of $C_r^*(G)$ and amenability of $G$ is known to hold in more general contexts than SIN groups. For example, Lau and Paterson showed this equivalence holds within the more general class of inner amenable groups (see \cite[Corollary 3.2]{lp}).

\begin{prop}
Let $G$ be a SIN group. The following are equivalent:
\begin{enumerate}
\item $G$ is amenable
\item $C^*_r(G)$ is nuclear
\item $C^*_r(G)\otimes_{\min}C^*_r(G)=C^*_r(G)\otimes_{\max}C^*_r(G)$ canonically
\item $\sigma_{\lambda_{G_d},\rho_{G_d}}$ extends to a $*$-representation of $C^*_{\lambda_G}(G_d)\otimes_{\min}C^*_{\lambda_G}(G_d)$
\end{enumerate}
\end{prop}

\begin{proof}
(1) $\Rightarrow$ (2): This is well known.

(2) $\Rightarrow$ (3): This is trivial.

(3) $\Rightarrow$ (4): For each compact neighbourhood $K$ of the identity which is invariant under inner automorphisms, define $f_K=1_K/m(K)^{1/2}\in L^2(G)$. Using a similar but simpler argument as that used in the proof of Theorem \ref{factorization property}, we will show that $(\sigma_{\lambda_G,\rho_G})_{f_K,f_K}$ converges pointwise to $1_\Delta$.

Observe that for $s_1,s_2\in G$ we have 
\begin{eqnarray*}
(\sigma_{\lambda_G,\rho_G})_{f_K,f_K}(s_1,s_2) &=& \frac{1}{m(K)}\int 1_K(s_1^{-1}ts_2)1_K(t)\,dt \\
&=& \frac{1}{m(K)}\int 1_{s_1K s_2^{-1}}(t)1_K(t)\,dt\\
&=& \frac{1}{m(K)}\int 1_{s_1s_2^{-1}  K }(t)1_K(t)\,ds.
\end{eqnarray*}
So if $s_1\neq s_2$ then, by taking $K$ small enough that $s_2s_1^{-1}\not\in K$, we have that $(\sigma_{\lambda_G,\rho_G})_{f_K,f_K}(s_1,s_2)=0$. Further, if $s_1=s_2$ then $(\sigma_{\lambda_G,\rho_G})_{f_K,f_K}(s_1,s_2)=\frac{1}{m(K)}\int 1_K(t)^2\,dt=1$. Hence, $(\sigma_{\lambda_G,\rho_G})_{f_K,f_K}$ converges pointwise to $1_\Delta$ as $K\to \{e\}$. It follows that $\sigma_{\lambda_{G_d},\rho_{G_d}}$ extends to a $*$-representation of $C^*_{\lambda_G}(G_d)$.

(4) $\Rightarrow$ (1): This is clear from Proposition \ref{factorization prop}.
\end{proof}

\section{Local Properties of Exotic Group C*-algebras}

In this section, we show that a large class of exotic group C*-algebras have poor local properties. In particular we show that these C*-algebras are not locally reflexive and, hence, not exact, and do not have the local lifting property.
This class of C*-algebras we consider are those of the form $C^*_{\pi}(\Gamma)$ where $\Gamma$ is a non-amenable discrete group and $\pi$ is a representation of $\Gamma$ such that:
\begin{enumerate}
\item[(i)] $\Gamma$ embeds algebraically into an amenable group $G$,
\item[(ii)] $\pi\otimes(\lambda_G|_\Gamma)$ is weakly contained in $\pi$.
\end{enumerate}
We will additionally stipulate that either $C^*_\pi(\Gamma)\neq C^*_r(\Gamma)$ canonically or $C^*_\pi(\Gamma)\neq C^*(\Gamma)$ canonically, depending on the theorem. Exotic group C*-algebras satisfying conditions (i) and (ii) are always included.

We note that for such C*-algebras $C^*_\pi(\Gamma)$ that $C^*_r(\Gamma)$ is necessarily a canonical quotient of $C^*_\pi(\Gamma)$. Indeed, suppose that $\Gamma$ and $\pi$ satisfy the conditions above. Then, since $\lambda_G|_\Gamma$ is known to weakly contain $\lambda$ (see \cite[Lemma 2]{be}), the representation $\pi$ must weakly contain the left regular representation as $1\otimes\lambda_\Gamma\cong \pi\otimes \lambda_\Gamma\prec \pi\otimes\lambda_G|_\Gamma\prec \pi$, where $\prec$ denotes weak containment.

Though our list of criterion for $C^*_\pi(\Gamma)$ may seem a little odd at first, we observe below that it includes many interesting examples of exotic group C*-algebras.

\begin{example}\ 
\begin{enumerate}
\item Suppose $\Gamma$ is a non-amenable discrete group which embeds algebraically into an amenable group $G$ and $D$ an algebraic ideal of $\ell^\infty(\Gamma)$ such that $C^*_D(\Gamma)\neq 0$, i.e., such that $\Gamma$ admits a $D$-representation. Then $C^*_D(\Gamma)$ is such an example. Indeed, let $\pi$ be a $D$-representation of $\Gamma$ such that $\pi$ extends to a faithful $*$-representation of $C^*_D(\Gamma)$ or, equivalently, $C^*_\pi(\Gamma)=C^*_D(\Gamma)$ canonically. Then $\pi\otimes\lambda_G|_\Gamma$ is weakly contained in $\pi$ since the tensor product of any $D$-representation with an arbitrary representation is a $D$-representation. So indeed, $C^*_D(\Gamma)$ is then an example of this form.
As particular cases of this type, we point out the following:
\begin{enumerate}
\item Let $\Gamma=\F_d$ be a noncommutative free group on $2\leq d\leq\infty$ generators. Then $C^*_{\ell^p}(\Gamma)$ is such an exotic group C*-algebra for $2<p<\infty$ since $\Gamma$, being residually finite, embeds into a compact group (recall that residually finite groups are maximally almost periodic, and maximally periodic group embed into compact groups, see \cite[Chapter 16]{d}).
\item Let $\Gamma$ be $SL_n(\Z)$ for $n\geq 3$ or $Sp_n(\Z)$ for $n\geq 2$. Then $C^*_{c_0}(\Gamma)$ is an exotic group C*-algebra of this type since
\begin{itemize}
\item $\Gamma$ is residually finite,
\item $C^*_{c_0}(\Gamma)\neq C^*(\Gamma)$ canonically as $\Gamma$ does not have the Haagerup property,
\item $C^*_{c_0}(\Gamma)\neq C^*_r(\Gamma)$ since $c_0(\Gamma)$ contains $\ell^p(\Gamma)$ for every $1\leq p<\infty$ and $\Gamma$ contains a copy of $\F_2$.
\end{itemize}
\end{enumerate}
\item Let $\Gamma$ be a non-amenable maximally almost periodic discrete group and $\mc F$ be the direct sum of all finite dimensional representations of $\Gamma$. Then $\Gamma$ embeds into a compact group $G$ and $\mc F\otimes \lambda_G|_\Gamma$ is weakly contained in $\mc F$ since $\lambda_G$ decomposes into a direct sum of finite dimensional representations. Further, $\mc F$ is not equivalent to $\lambda$ since $\mc F$ contains the trivial representation. So $C^*_{\mc F}(\Gamma)$ is an example of this type. It is clear that $C^*_{\mc F}(\Gamma)$ is an exotic group C*-algebra if and only if $C^*(\Gamma)$ fails to be residually finite dimensional. Since Bekka has demonstrated a large class of residually finite groups whose full group C*-algebras are not residually finite dimensional (see \cite{b}), there is no shortage of examples of this type. Some examples include $SL_n(\Z)$ for $n\geq 3$ and $Sp_n(\Z)$ for $n\geq 2$.
\item Let $G$ be a locally compact group and $\Gamma=G_d$ be the group $G$ endowed with the discrete topology. Then $C^*_r(\Gamma)\neq C^*_{\lambda_G}(G_d)$ if $G_d$ is non-amenable since the trivial representation extends to a $*$-representation of $C^*_{\lambda_G}(G_d)$ but not of $C^*_r(G_d)$. C*-algebras of this type have been studied by Bekka, Kaniuth, Lau and Schlichting in \cite{bkls} when $G$ is not necessarily amenable. Determining when $C^*_{\lambda_G}(G_d)= C^*(G_d)$ canonically remains an open problem, but it is conjectured that this should hold if and only if $G$ admits an open subgroup $H$ such that $H_d$ is amenable (see \cite[Remark 1]{bkls}).
The conjecture is known to hold when $G$ is a connected Lie group (see \cite{bv}) and, so, this construction gives an exotic group C*-algebra, for example, when $G=SO_n(\R)$ for $n\geq 3$ or $SU_n(\C)$ for $n\geq 2$. We note that the C*-algebra $C^*_{\lambda_G}(G_d)$ is also denoted by $C^*_\delta(G)$ in the literature.
\end{enumerate}
\end{example}

We will now demonstrate the failure of local properties for this class of C*-algebras.

\subsection*{Local Reflexivity}

Let us first recall that a  C*-algebra (or an operator space) $A$  is {\em locally reflexive} if for any finite dimensional operator space $E$ and complete contraction $T : E \to A^{**}$, there exists a net of complete contractions $T_i : E \to A$ such that $T_i \to T$ in point-weak* topology. 
Local reflexivity is a very important operator space local property which is equivalent to Property $C^{\prime \prime}$ introduced by Archbold and Batty (see \cite{ab}). We will show that the class of C*-algebras we are considering fail to be locally reflexive. Towards this goal, we recall that locally reflexive C*-algebras satisfy the following property (see \cite[Proposition 5.3]{eh} or \cite[Proposition 14.3.7]{er}).

\begin{proposition} \label{localthm}
If $B$ is a locally reflexive C*-algebra, then the sequence
\[
0 \to J \otimes_{\rm min} C \to B \otimes_{\rm min} C \to B/J \otimes_{\rm min}  C \to 0
\]
is exact for every closed two-sided ideal $J$ of $B$ and every C*-algebra $C$.
\end{proposition}

Let $A \subseteq  B(H)$ be a unital C*-algebra on a Hilbert space $H$. 
A trace $\tau $ on $A$ is called an {\em amenable trace} if there is a state $\varphi$ on $B(H)$ such that $\varphi|_A = \tau$ and $\varphi(u T u^*) = \varphi(T)$ for all unitary operators $u\in A$ and $T \in B(H)$.
We note that the definition of amenable trace is independent from the choice of Hilbert space $H$ and a trace $\tau$ on $A$ is an amenable trace if and only if the positive linear functional 
\[
\mu_\tau : a \otimes b^{\tn{op}} \in A \otimes A^{\tn{op}} \mapsto  \tau(ab) \in \mathbb C
\]
is continuous with respect to the minimal tensor product norm (see \cite [ Theorem 6.2.7]{bo}).

The authors are grateful to an anonymous referee for suggesting the inclusion of the following Theorem which generalizes Corollary \ref{nonlocalcor}.

\begin{theorem}\label{nonlocal}
Let $\Gamma$ be a non-amenable discrete group with the factorization property and $\sigma$ a representation of $\Gamma$ which weakly contains $\lambda$ such that the canonical trace $\tau_\lambda$ is an amenable trace of $C^*_\sigma(\Gamma)$. If $\pi$ is a representation which weakly contains but is not weakly equivalent to $\lambda$, then $C^*_{\pi\otimes\sigma}(\Gamma)$ is not locally reflexive.
\end{theorem}

\begin{proof}



Since $\tau_\lambda$ is an amenable trace of $C^*_\sigma(\Gamma)$, the functional on $C^*_\sigma(\Gamma)\odot C^*_\sigma(\Gamma)^{\mathrm{op}}$ defined by $a\otimes b^{\mathrm{op}}\mapsto \tau_\lambda(ab)$ is continuous with respect to the minimal C*-norm (see \cite[Theorem 6.2.7]{bo}). Let $\sigma^{\mathrm{op}}\fn G\to C^*_\sigma(\Gamma)^{\mathrm{op}}$ denote the representation defined by $\sigma^{\mathrm{op}}(s)=\sigma(s^{-1})^\mathrm{op}$. Since $C^*_{\sigma^\mathrm{op}}(\Gamma)=C^*_{\sigma}(\Gamma)^\mathrm{op}$, we then have that the map $s\otimes t\mapsto \tau_\lambda(st^{-1})=1_\Delta(s,t)$ (for $s,t\in\Gamma$) extends a bounded linear functional of $C^*_\sigma(\Gamma)\otimes_{\min}C^*_{\sigma^{\mathrm{op}}}(\Gamma)$. So $\pi_{\lambda,\rho}$, by virtue of being the GNS representation of $1_\Delta$, extends to a $*$-representation of $C^*_\sigma(\Gamma)\otimes_{\min}C^*_{\sigma^{\mathrm{op}}}(\Gamma)$ and, hence, the representation $\mu\fn \Gamma\times\Gamma\to B(H_\pi\otimes \ell^2(\Gamma))$ defined by
$$\mu(s,t)=\pi(s)\otimes \lambda(s)\rho(t)$$
extends to $*$-representation of $C^*_{\pi\otimes\sigma}(\Gamma)\otimes_{\min} C^*_{\sigma^{\mathrm{op}}}(\Gamma)$. 

Recall that since $\pi$ weakly contains $\lambda$, the representation $\pi\otimes \sigma$ weakly contains $\lambda\otimes \sigma$ and thus weakly contains $\lambda$ since $\lambda \otimes \sigma$ is unitarily equivalent to an amplification of $\lambda$ by Fell's absorption principle. 
In this case,  we have $\|\pi\otimes \sigma(\cdot)\| \ge \|\lambda(\cdot)\|$ on ${\mathbb C}[\Gamma]$  (see \cite[Proposition 18.1.4]{d}) and thus 
the representation $\lambda$ extends canonically  to a $*$-representation $\widetilde\lambda\fn C^*_{\pi\otimes\sigma}(\Gamma)\to C^*_r(\Gamma)$. 
Observe that $\|\widetilde \lambda (\cdot)\| \ge\|\mu|_{\Gamma\times \{e\}} (\cdot)\|$ on $C^*_{\pi\otimes\sigma}(\Gamma)$ since 
$\mu|_{\Gamma\times \{e\}}=\pi\otimes \lambda$ is unitarily equivalent to an amplification of $\lambda$. Therefore 
$\ker\mu|_{\Gamma\times \{e\}}$ contains $\ker\widetilde\lambda$ in
$C^*_{\pi \otimes \sigma}(\Gamma)$, and thus
$\ker\mu$ contains $\ker\widetilde\lambda\otimes_{\min} C^*_{\sigma^{\mathrm{op}}}(\Gamma)$ 
when $\mu$ is viewed as a $*$-representation of $C^*_{\pi\otimes\sigma}(\Gamma)\otimes_{\min} C^*_{\sigma^{\mathrm{op}}}(\Gamma)$.

Suppose towards a contradiction that $C^*_{\pi\otimes\sigma}(\Gamma)$ is locally reflexive. Then 
$$\ker\big(\widetilde\lambda\otimes\mathrm{id}\fn C^*_{\pi\otimes\sigma}(\Gamma)\otimes_{\min} C^*_{\sigma^{\mathrm{op}}}(\Gamma)\to C^*_r(\Gamma)\otimes_{\min}C^*_{\sigma^{\mathrm{op}}}(\Gamma)\big) = \ker\widetilde\lambda\otimes_{\min} C^*_{\sigma^{\mathrm{op}}}(\Gamma)$$
by Proposition \ref{localthm} and, thus,  $\mu$ is weakly contained in $\lambda\times \sigma^{\mathrm{op}}$. So if $\Delta\fn \Gamma\to \Gamma\times\Gamma$ is the diagonal embedding, then $\mu\circ \Delta$ is weakly contained in $(\lambda\times \sigma^{\mathrm{op}})\circ \Delta=\lambda\otimes \sigma^{\mathrm{op}}\cong\lambda^{\oplus\alpha}$ for some cardinal $\alpha$. But $\mu\circ\Delta$ contains $\pi$ as a subrepresentation which contradicts our assumption that $\pi$ is not weakly contained in $\lambda$. Thus, we conclude that $C^*_{\pi\otimes\sigma}(\Gamma)$ is not locally reflexive.
\end{proof}

\begin{corollary}\label{nonlocalcor}
Suppose $\Gamma$ and $\pi$ satisfy the conditions prescribed at the beginning of the section and that $C^*_\pi(\Gamma)\neq C^*_r(\Gamma)$ canonically. Then $C^*_\pi(\Gamma)$ is not locally reflexive.
\end{corollary}

\begin{proof}
Let $G$ be a locally compact group which satisfies the prescribed conditions. Since $\lambda_G$ is unitarily equivalent to $\lambda_G^\mathrm{op}$, the functional $1_\Delta$ extends to a positive linear functional of $C^*_{\lambda_G|_\Gamma}(\Gamma)\otimes_{\min} C^*_{(\lambda_G|_\Gamma)^\mathrm{op}}(\Gamma)$ by Theorem \ref{factorization property}. Hence, the canonical trace $\tau_\lambda$ is an amenable trace of $C^*_{\lambda_G|_\Gamma}(\Gamma)$ (see \cite[Theorem 6.2.7]{bo}). So, with $\lambda_G|_\Gamma$ playing the roll of $\sigma$, $C^*_{\pi\otimes \lambda_G|_\Gamma}$ is not locally reflexive by the previous lemma. 
 As $\pi\otimes\lambda_G|_\Gamma$ is weakly contained in $\pi$ and local reflexivity is preserved under C*-quotients, we conclude that $C^*_\pi(\Gamma)$ is not locally reflexive.
%
\end{proof}

\subsection*{Exactness}

Let us  recall from Kirchberg's definition
that a C*-algebra $A$ is {\em exact}  if 
for any C*-algebra $B$, and any  norm closed two-sided ideal $J \subseteq B$, the sequence
\[
0\to A \otimes_{\rm min} J \to A \otimes_{\rm min} B \to A \otimes_{\rm min} B/J \to 0
\]
is exact.
It is very often (for instance, when $\Gamma = {\mathbb F}_d$ for $d \ge 2$,  ${\mathbb Z}^2 \rtimes SL_2({\mathbb Z})$, or $SL_n({\mathbb Z})$
for $n \ge 2$) 
that the reduced group C*-algebra $C^*_r(\Gamma)$ is  exact, but the full group C*-algebra $C^*(\Gamma)$ is not.
Indeed,  it is known that if $\Gamma$ is a residually finite discrete group, then   the full group C*-algebra
$C^*(\Gamma)$ is not exact  if and only if the group $\Gamma$ is non-amenable  (see \cite[Proposition 3.7.11]{bo}).
We now show the failure of exactness in the class of C*-algebras we are considering.

\begin{theorem} \label{nonexact}
Suppose $\Gamma$ and $\pi$ satisfy the conditions prescribed at the beginning of the section and that $C^*_\pi(\Gamma)\neq C^*_r(\Gamma)$ canonically. Then $C^*_\pi(\Gamma)$ is not exact.
\end{theorem} 
\begin{proof}
This is an immediate consequence of Theorem \ref {nonlocal} since every exact C*-algebra (or operator space) is locally reflexive (see \cite{eor}). 
Alternatively, we can apply the same argument as given in the proof of Theorem \ref {nonlocal}  since the conclusion for Proposition \ref{localthm} also holds when C*-algebra $B$  is exact (see \cite[Proposition 5.3]{eh}).
\end{proof}

\subsection*{Local Lifting Property}

A unital C*-algebra $A$ has the  {\em local lifting property} (or in short {\em LLP}) if for any C*-algebra $B$ with closed ideal $J$ and any unital completely positive map $\varphi : A \to B/J$ is locally liftable, i.e., for any finite dimensional operator system $E \subseteq A$, there exists a unital completely  positive map $\psi : E \to B$ such that $\pi \circ \psi = \varphi|_E$.
It is known that the full group C*-alegbra $C^*(\F_2)$ has the lifting property 
and thus has the LLP.
We will make use of the following characterization of the LLP due to Effros and Haagerup (see \cite{eh}) to show failure of LLP in the class of C*-algebras under consideration.

\begin{proposition}
A unital C*-algebra $A = B/J$ has the LLP if and only if for any C*-algebra $C$, the sequence
\[
0 \to  J \otimes _{\rm min} C \to B \otimes _{\rm min} C \to B/J \otimes _{\rm min} C \to 0
\]
is exact.\end{proposition}

We remark that in the previous two theorems we assumed that $C^*_\pi(\Gamma)\neq C^*_r(\Gamma)$ but allowed $C^*_\pi(\Gamma)=C^*(\Gamma)$. Our next result reverses these roles of $C^*_r(\Gamma)$ and $C^*(\Gamma)$ in the sense that $C^*_\pi(\Gamma)$ is allowed to be $C^*_r(\Gamma)$ but not $C^*(\Gamma)$.

\begin{theorem}\label{llp}
Suppose $\Gamma$ and $\pi$ satisfy the conditions prescribed at the beginning of the section and that $C^*_\pi(\Gamma)\neq C^*(\Gamma)$ canonically. Then $C^*_\pi(\Gamma)$ does not have the LLP.
\end{theorem}


\begin{proof}

By Theorem \ref{factorization property}, the representation $\sigma_{\lambda_\Gamma,\rho_\Gamma}$ extends to a $*$-representation of $C^*_{\lambda_G|_\Gamma}(\Gamma)\otimes_{\min} C^*_{\lambda_G|_\Gamma}$. Let $\pi_u$ denote the universal representation of $\Gamma$. Then $\pi_u\otimes\lambda_G|_\Gamma$ is weakly equivalent to $\pi_u$ since $\lambda_G|_\Gamma$ weakly contains the trivial representation and, hence, the representation $\mu\fn \Gamma\times\Gamma\to B(H_u)\otimes B(\ell^2(\Gamma))$ defined by
$$ \mu(s,t)=\pi_u(s)\otimes \lambda_\Gamma(s)\rho_\Gamma(t) $$
extends to a $*$-representation of $C^*(\Gamma)\otimes_{\min} C^*_{\lambda_G|_\Gamma}(\Gamma)$.

Suppose towards a contradiction that $C^*_\pi(\Gamma)$ has the LLP. Then
$$\ker\big(\pi\otimes\mathrm{id}\fn C^*(\Gamma)\otimes_{\min} C^*_{\lambda_G|_\Gamma}(\Gamma)\to C^*_{\pi}(\Gamma)\otimes_{\min}C^*_{\lambda_G|_\Gamma}(\Gamma)\big)= \ker\pi\otimes_{\min} C^*_{\lambda_G|_\Gamma}(\Gamma)$$
by the above proposition. Observe that when $\mu$ is viewed as a $*$-representation of $C^*(\Gamma)\otimes_{\min} C^*_{\lambda_G|_\Gamma}(\Gamma)$ that the kernel of $\mu$ contains $\ker\lambda_\Gamma\otimes_{\min} C^*_{\lambda_G|_\Gamma}(\Gamma)$ since $\mu|_{\Gamma\times \{e\}}=\pi_u\otimes \lambda_\Gamma$ is unitarily equivalent to an amplification of $\lambda_\Gamma$ by Fell's absorption principle. 
Therefore $\mu$ is weakly contained in $\pi\times\lambda_G|_\Gamma$ by the above condition implied by the LLP since $\ker \lambda_\Gamma\supset \ker \pi$ when $\lambda_\Gamma$ and $\pi$ are viewed as $*$-representations of $C^*(\Gamma)$.

Let $\Delta\fn \Gamma\to\Gamma\times\Gamma$ denote the diagonal embedding. Observe that $\mu\circ \Delta$ contains $\pi_u$ as a subrepresentation. But $(\pi\times\lambda_G|_{\Gamma})\circ \Delta=\pi\otimes \lambda_G|_\Gamma\prec \pi$ does not weakly contain $\pi_u$ and, so, we arrive at a contradiction. Therefore $C^*_\pi(\Gamma)$ does not have the LLP.
\end{proof}

\begin{remark}
Let $\Gamma$ be a discrete group and $\pi$  a representation of $\Gamma$.
If $H$ is a subgroup of $\Gamma$ such that $C^*_{\pi|_{H}}(H)$ is an exotic group C*-algebra of $H$ and 
$\pi|_{H}$  satisfies  the conditions prescribed at the beginning of this section, 
then the results of this section also apply to $C^*_\pi(\Gamma)$ since $C^*_{\pi|_{H}}(H)$ is a C*-subalgebra of 
$C^*_\pi(\Gamma)$ and  local reflexivity, exactness and LLP are all local properties. 
In particular, this implies that $C^*_{\ell^p}(\Gamma)$ for $2<p<\infty$ fails to have these properties where $\Gamma$ is a discrete group containing a noncommutative free group since $C^*_{\ell^p}(\F_2)$ embeds into $C^*_{\ell^p}(\Gamma)$.
\end{remark}

\begin{remark}
We note that not every exotic group C*-algebra need lack all these properties. Indeed, consider the group $\F_2$ and let $\sigma=\lambda\oplus 1_{\F_2}$. Then $C^*_\sigma(\F_2)$ is an exotic group C*-algebra and it is easily verified that $C^*_\sigma(\F_2)\cong C^*_r(\F_2)\oplus \C$. It follows that $C^*_{\sigma}(\F_2)$ has the CCAP by virtue of $C^*_r(\F_2)$ having this property.
\end{remark}


\section{Additional Remarks on $D$-C*-algebras}

We finish off the paper by remarking on a few further properties of $D$-C*-algebras for discrete groups. We will no longer insist that $\Gamma$ embeds inside an amenable locally compact group $G$.

\subsection*{Amenable Traces}

Recall that if $\Gamma$ is a discrete group, then $\Gamma$ is amenable if and only if $C^*_r(\Gamma)$ has an amenable trace. In fact any trace on $C^*_r(\Gamma)$ is amenable in this case. The following proposition shows that this remains true when $C^*_r(\Gamma)$ is replaced by $C^*_D(\Gamma)$ when $C^*_D(\Gamma)\neq C^*(\Gamma)$.
This contrasts with the fact that the trace on $C^*(\Gamma)$ coming from the trivial representation is always amenable (see \cite[Exercise 6.3.2]{bo}).

\begin{lemma}\label{self-adjoint}
Let $\Gamma$ be a discrete group (or an arbitrary set) and $D$ an algebraic ideal of $\ell^\infty(\Gamma)$. Then $D$ is self-adjoint (i.e., $\overline{f}\in D$ for every $f\in D$).
\end{lemma}

\begin{proof}
Let $f\in D$ and define $g\in \ell^\infty(\Gamma)$ by $g(s)=\overline{\mathrm{sgn}\,f(s)}$ for $s\in\Gamma$. Then $\overline{f}=fg^2\in D$ since $D$ is an ideal of $\ell^\infty(\Gamma)$.
\end{proof}

\begin{prop}\label{trace}
Suppose that $\Gamma$ is a non-amenable discrete group and $D$ is an algebraic ideal of $\ell^\infty(\Gamma)$ such that $C^*_D(\Gamma)\neq 0$ and $C^*_D(\Gamma)\neq C^*(\Gamma)$ canonically. Then $C^*_D(\Gamma)$ does not admit any amenable traces. 
\end{prop}

\begin{proof}
Suppose that $C^*_{D}(\Gamma)$ admits an amenable trace $\tau$. Then the map $a\otimes b^{\tn{op}}\mapsto \tau (ab)$ extends to a state on $C^*_{D}(\Gamma)\otimes_{\min} C^*_{D}(\Gamma)^{\tn{op}}$. Observe that if $\pi$ is an $D$-representation if and only if $\overline{\pi}$ is an $D$-representation since $D$ is self-adjoint. So $C^*_{D}(\Gamma)^{\tn{op}}\cong C^*_{D}(\Gamma)$ via $s^{\tn{op}} \mapsto s^{-1}$ since for all $a_1,\ldots,a_n\in \C$ and $s_1,\ldots,s_n\in \Gamma$,
\begin{eqnarray*}
\bigg\|\sum_{i=1}^n{a_i}s^{-1}\bigg\|_{D} &=& \sup_{\pi,\xi,\eta}\bigg|\sum_{i=1}^n a_i\lla \pi(s^{-1})\xi,\eta\rra\bigg|\\
&=&\sup_{\pi,\xi,\eta}\bigg|\sum_{i=1}^n a_i\lla \xi,\pi(s)\eta\rra\bigg|\\
&=&\sup_{\pi,\xi,\eta}\bigg|\sum_{i=1}^n a_i\lla \overline{\pi}(s)\overline{\eta},\overline{\xi}\rra\bigg|\\
&=&\bigg\|\sum_{i=1}^na_is_i\bigg\|_{D}
\end{eqnarray*}
where the supremums are taken over $D$-representations $\pi\fn \Gamma\to B(H_\pi)$ and $\xi,\eta\in H_\pi$ with $\|\xi\|\leq 1$ and $\|\eta\|\leq 1$. This in particular implies that the map $$s\otimes t\mapsto\tau(st^{-1})$$ for $s,t\in \Gamma$ extends to a positive linear functional on $C^*_{D}(\Gamma)\otimes_{\min} C^*_{D}(\Gamma)$. When restricted to the diagonal subgroup $\Delta$, this function is identically 1.

Let $\pi$ be a $D$-representation which extends to a faithful representation of $C^*_D(\Gamma)$. Then the above observation implies that the representation $\pi\otimes\pi$ of $\Gamma$ weakly contains the trivial representation. Hence, $\pi$ weakly contains the trivial representation since $\pi\otimes\pi$ is weakly contained in $\pi$. So $\Gamma$ is amenable (see \cite[Theorem 3.2]{bg}).
\end{proof}

\subsection*{Quasidiagonality and WEP}

Recall that a C*-algebra $A$ is {\it quasidiagonal} if there exists a net of completely contractive positive (c.c.p.) maps $\varphi_i\fn A\to M_{n(i)}$ into finite dimensional C*-algebras which are asymptotically multiplicative and isometric in the sense that $\|\varphi_i(ab)-\varphi_i(a)\varphi_i(b)\|\to 0$ for all $a,b\in A$ and $\|\varphi_i(a)\|\to \|a\|$ for every $a\in A$.
Also, recall that a C*-algebra $A$ has the {\it weak expectation property} ({\it WEP}) if there exists a c.c.p. map $\Phi\fn B(H)\to A^{**}$ such that $\Phi(a)=a$ for all $a\in A$, where $A\subseteq B(H)$ is some faithful representation of $A$. This property is independent of the choice of faithful representation.
As a consequence of the previous theorem, we observe the failure of quasidiagonality and WEP in $D$-C*-algebras.

\begin{corollary}\label{quasi and wep}
Suppose that $\Gamma$ is a non-amenable discrete group and $D$ is an algebraic ideal of $\ell^\infty(\Gamma)$ such that $C^*_D(\Gamma)\neq 0$ and $C^*_D(\Gamma)\neq C^*(\Gamma)$ canonically. Then
\begin{itemize}
\item[(i)] $C^*_{D}(\Gamma)$ is not quasidiagonal;
\item[(ii)] $C^*_{D}(\Gamma)$ does not have the WEP.
\end{itemize}
\end{corollary}

\begin{proof}
(i): Quasidiagonal C*-algebras are known to admit an amenable trace (see \cite[Proposition 7.1.16]{bo}). Therefore $C^*_{D}(\Gamma)$ is not quasidiagonal as it has no amenable traces.

(ii): It is well known that every trace on a C*-algebra with the WEP is amenable (see \cite[Exercise 6.3.1]{bo}). Since the canonical trace on $C^*_r(\Gamma)$ gives rise to a trace on $C^*_{D}(\Gamma)$, we conclude that $C^*_{D}(\Gamma)$ fails to admit the WEP.
\end{proof}

\section*{Acknowledgements}
The authors are grateful to an anonymous referee for the suggestion to include Theorem \ref{nonlocal} as a generalization of Corollary \ref{nonlocalcor} and an observation which helped shorten the proof of Theorem \ref{factorization property}.



\begin{thebibliography}{9}

\bibitem{ab}
R. Archbold and C. Batty, {\it $C^*$-tensor norms and slice maps}. J. London Math. Soc., {\bf 22} (1980), 127-138.


\bibitem{be}
E. B\'edos, \emph{On the $C^*$-algebra generated by the left regular representation of a locally 
compact group}. Proc. Amer. Math. Soc. {\bf 120} (1994), no. 2, 603-608.

\bibitem{b}
M.B. Bekka, {\it On the full $C^*$-algebras of arithmetic groups and the congruence subgroup problem}. Forum Math. {\bf 11} (1999), no. 6, 705-715.

\bibitem{bdv}
M.B. Bekka, P. de la Harpe and A. Valette, {\it Kazhdan's property (T)}. New Mathematical Monographs, {\bf 11}. Cambridge University Press, Cambridge, 2008.

\bibitem{bkls}
M.B. Bekka, E. Kaniuth, A.T. Lau and G. Schlichting, {\it On $C^*$-algebras associated with locally compact groups}. Proc. Amer. Math. Soc. {\bf 124} (1996), no. 10, 3151-3158.

\bibitem{bv}
M.B. Bekka and A. Valette, {\it On duals of Lie groups made discrete}. J. Reine Angew. Math. {\bf 439} (1993), 1-10.

\bibitem{br}
M. Brannan and Z.-J. Ruan, {\it $L_p$-representations of discrete quantum groups}. J. Reine Angew. Math., to appear.

\bibitem{bg}
N.P. Brown and E.P. Guentner, \emph{New $C^*$-completions of discrete groups and related spaces}. Bull. Lond. Math. Soc. {\bf 45} (2013), no. 6, 1181-1193.

\bibitem{bo}
N.P. Brown and N. Ozawa, \emph{$C^*$-algebras and finite-dimensional approximations}.
Graduate Studies in Mathematics, {\bf 88}. American Mathematical Society, Providence, RI, 2008.

\bibitem{d}
J. Dixmier, \emph{$C^*$-algebras}. North-Holland Mathematical Library, Vol. 15. North-Holland Publishing Co., Amsterdam-New York-Oxford, 1977.

\bibitem{eh}
E.G. Effros and U. Haagerup, {\it Lifting problems and local reflexivity for $C^*$-algebras}.
Duke Math. J. {\bf 52} (1985), no. 1, 103-128. 

\bibitem{el}
E.G. Effros and E.C. Lance, {\it Tensor products of operator algebras}. Adv. Math. {\bf 25} (1977), no. 1, 1-34.

\bibitem{eor}
E.G. Effros, N. Ozawa, Z.J. Ruan, {\it On injectivity and nuclearity for operator spaces}. Duke Math. J. {\bf 110} (2001), no. 3, 489-521.

\bibitem{er}
E.G. Effros and Z.-J. Ruan, {\it Operator spaces}. London Mathematical Society Monographs. New Series, {\bf 23}. The Clarendon Press, Oxford University Press, New York, 2000.


\bibitem{k}
E. Kirchberg, \emph{On nonsemisplit extensions, tensor products and exactness of group $C^*$-algebras}. Invent. Math. {\bf 112} (1993), no. 3, 449-489.

\bibitem{ks}
D. Kyed and P.M. So\l{}tan, {\it Property (T) and exotic quantum group norms}.
J. Noncommut. Geom. {\bf 6} (2012), no. 4, 773-800.

\bibitem{lp}
A.T.-M. Lau and A.L.T. Paterson, {\it Inner amenable locally compact groups}. Trans. Amer. Math. Soc. {\bf 325} (1991), no. 1, 155-169.

\bibitem{lr}
V. Losert and H. Rindler, {\it Almost invariant sets}. Bull. London Math. Soc. {\bf 13} (1981), no. 2, 145-148.

\bibitem{o}
R. Okayasu, {\it Free group $C^*$-algebras associated with $\ell^p$}. Internat. J. Math. {\bf 25} (2014), no. 7, 1450065, 12 pp.


\bibitem{thom}
A. Thom, {\it Examples of hyperlinear groups without factorization property}. Groups Geom. Dyn. {\bf 4} (2010), no. 1, 195-208.

\bibitem{wa}
S. Wassermann, \emph{On tensor products of certain group $C^*$-algebras}. J. Functional Analysis {\bf 23} (1976), no. 3, 239-254.


\bibitem{w1}
M. Wiersma, \emph{Constructions of exotic group $C^*$-algebras}. Preprint (2014), see arXiv:1403.5223v1.

\end{thebibliography}
\end{document}